\documentclass[a4paper,reqno,10pt]{amsart}
\usepackage{graphicx,verbatim, amsmath, amssymb, amsthm, amsfonts, epsfig, ifthen,mathtools,
caption,enumerate,mathrsfs,url}	
\usepackage[all]{xy}
\xyoption{all}
\usepackage{hyperref} 

\newtheorem{proposition}{Proposition}[section]
\newtheorem{theorem}[proposition]{Theorem}
\newtheorem{corollary}[proposition]{Corollary}
\newtheorem{lemma}[proposition]{Lemma}

\theoremstyle{definition}
\newtheorem{example}[proposition]{Example}
\newtheorem{definition}[proposition]{Definition}
\newtheorem{notation}[proposition]{Notation}

\theoremstyle{remark}
\newtheorem{remark}[proposition]{Remark}

\numberwithin{equation}{section}
\usepackage[usenames]{color}

\usepackage{color}

\newcommand{\margincolor}{red}      
\definecolor{darkgreen}{rgb}{0,0.7,0}

\newcounter{margincounter}
\newcommand{\marginnum}{
\ifnum\value{margincounter}<10
\textcolor{\margincolor}{\begin{picture}(0,0)\put(2.2,2.4){\circle{9}}\end{picture}\footnotesize\arabic{margincounter}}
\else\ifnum\value{margincounter}<100
\textcolor{\margincolor}{\begin{picture}(0,0)\put(4.256,2.5){\circle{11}}\end{picture}\footnotesize\arabic{margincounter}}
\else
\textcolor{\margincolor}{\begin{picture}(0,0)\put(6.8,2.5){\circle{14}}\end{picture}\footnotesize\arabic{margincounter}}
\fi\fi
}

\newcommand{\integers}{\mathbb Z}

\newcommand{\naturals}{\mathbb N}
\newcommand{\reals}{\mathbb R}

\DeclareMathOperator{\brick}{\mathsf{brick}}

\DeclareMathOperator{\tors}{\mathsf{tors}}

\DeclareMathOperator{\torf}{\mathsf{torf}}

\DeclareMathOperator{\sbrick}{\mathsf{sbrick}}

\DeclareMathOperator{\Sub}{\mathsf{Sub}}
\DeclareMathOperator{\Filt}{\mathsf{Filt}}

\DeclareMathOperator{\add}{\mathsf{add}}
\DeclareMathOperator{\Fac}{\mathsf{Fac}}

\let\mod\relax
\DeclareMathOperator{\mod}{\mathsf{mod}}

\DeclareMathOperator{\Hom}{\mathsf{Hom}}
\DeclareMathOperator{\End}{\mathsf{End}}

\DeclareMathOperator{\Kernel}{\mathsf{Ker}}
\DeclareMathOperator{\Cokernel}{\mathsf{Cok}}
\DeclareMathOperator{\T}{\mathsf{T}}
\DeclareMathOperator{\Image}{\mathsf{Im}}

\DeclareMathOperator{\proj}{\mathsf{proj}}
\DeclareMathOperator{\id}{\mathsf{id}}

\DeclareMathOperator{\Hasse}{\mathsf{Hasse}}

\DeclareMathOperator{\wide}{\mathsf{wide}}

\DeclareMathOperator{\Label}{\mathsf{Label}}

\newcommand{\join}{\vee}

\renewcommand{\Join}{\bigvee}
\newcommand{\Meet}{\bigwedge}
\newcommand{\xto}{\xrightarrow}

\newcommand{\op}{{\mathrm{op}}}

\newcommand{\CC}{{\mathcal C}}
\newcommand{\WW}{{\mathcal W}}
\newcommand{\TT}{{\mathcal T}}

\renewcommand{\AA}{{\mathcal A}}

\renewcommand{\S}{{\mathcal S}}

\newcommand{\UU}{\mathcal{U}}
\newcommand{\VV}{\mathcal{V}}
\newcommand{\FF}{{\mathcal F}}
\newcommand{\X}{{\mathcal X}}
\newcommand{\LL}{{\mathcal L}}

\DeclareMathOperator{\WL}{\mathsf{W}_L}
\DeclareMathOperator{\WR}{\mathsf{W}_R}
\DeclareMathOperator{\simple}{\mathsf{sim}}
\newcommand{\GG}{{\mathcal G}}
\DeclareMathOperator{\F}{\mathsf{F}}
\DeclareMathOperator{\Serre}{\mathsf{Serre}}

\title{Wide subcategories and lattices of torsion classes}
\author{Sota Asai} 
\address{Sota Asai: Research Institute for Mathematical Sciences, Kyoto University, Kitashirakawa-Oiwakecho, Sakyo-ku, Kyoto-shi, Kyoto-fu, 606-8502, Japan}
\email{asaisota@kurims.kyoto-u.ac.jp}
\author{Calvin Pfeifer}
\address{Calvin Pfeifer: Mathematisches Institut,
Universit\"{a}t Bonn, Endenicher Allee 60, 53115 Bonn, Germany}
\email{calvin.pfeifer@uni-bonn.de}
\date{\today}

\begin{document}

\begin{abstract}
In this paper, we study the relationship between wide subcategories and torsion classes of an abelian length category $\AA$ from the point of view of lattice theory.
Motivated by $\tau$-tilting reduction of Jasso,
we mainly focus on intervals $[\UU,\TT]$ in the lattice $\tors \AA$ of torsion classes in $\AA$ such that $\WW:=\UU^\perp \cap \TT$ is a wide subcategory of $\AA$; we call these intervals wide intervals. 
We prove that a wide interval $[\UU,\TT]$ is isomorphic to
the lattice $\tors \WW$ of torsion classes in the abelian category $\WW$. 
We also characterize wide intervals in two ways: First, in purely lattice theoretic terms based on the brick labeling established by Demonet--Iyama--Reading--Reiten--Thomas; 
and second, 
in terms of the Ingalls--Thomas correspondences 
between torsion classes and wide subcategories, which were
further developed by Marks--\v{S}\v{t}ov\'{i}\v{c}ek.
\end{abstract}

\maketitle

\tableofcontents

\section{Introduction}

In the mid-20th century Dickson \cite{D} vastly generalized the torsion theory of abelian groups to abelian categories $\AA$. He defined a \textit{torsion pair} $(\TT,\FF)$ in $\AA$ to be a pair of full subcategories $\TT, \FF \subseteq \AA$ satisfying the $\Hom$-orthogonality conditions 
\[
\FF=\TT^\perp \text{ and } \TT={^\perp\FF}.
\]
Here and throughout, we indicate the \textit{right $\Hom$-perpendicular category} of a full subcategory $\X\subseteq\AA$ by
\begin{align*}
\X^\perp:=\{Y\in\AA \mid \text{for all $X \in \X$, $\Hom_{\AA}(X,Y)=0$} \},
\end{align*}
and define the \textit{left $\Hom$-perpendicular category} $^\perp\X$ dually. Since then many authors investigated torsion pairs from several points of view, including their
classification \cite{Br,Ho},
derived equivalences \cite{BB,Ha,HRS,Rickard}, 
cluster theory \cite{IT},
$\tau$-tilting theory \cite{AIR, KY, MS}
and stability conditions \cite{B}.

From now on we assume that $\AA$ is an essentially small abelian length category.
For a full subcategory $\TT \subseteq \AA$ 
there is a torsion pair $(\TT,\FF)$ in $\AA$ if and only if
$\TT$ is closed under extensions and factor objects;
such subcategories $\TT$ are called \textit{torsion classes}.
In our setting the torsion classes in $\AA$ form a set $\tors \AA$, partially ordered by inclusion. Note that $\tors \AA$ is stable under arbitrary intersections. In particular, for every full subcategory $\X\subseteq\AA$, there is a smallest torsion class $\T(\X)$ containing $\X$, namely
\begin{align}\label{minimaltorsionclass}
\T(\X):=\bigcap_{\TT\in\tors \AA;\X\subseteq\TT}\TT.
\end{align}
It is not hard to see that the partially ordered set $\tors \AA$ is in fact a \textit{complete lattice},
that is, \textit{joins} and \textit{meets} of arbitrary subsets of $\tors \AA$ exist. More generally, every interval
\begin{align*}
[\UU,\TT]:=\{ \VV \in \tors \AA \mid \UU \subseteq \VV \subseteq \TT \}
\end{align*}
in $\tors \AA$ is a complete lattice. To an interval $[\UU,\TT]$ in $\tors \AA$ we associate a full subcategory $\WW:=\UU^\perp\cap\TT$ as proposed in \cite{DIRRT}. This subcategory ``measures the difference'' between $\UU$ and $\TT$ in a precise sense (see Lemma \ref{lemma:intervalcategory}). We investigate those intervals $[\UU,\TT]$ in $\tors \AA$ such that $\WW:=\UU^\perp \cap \TT$ is an abelian subcategory of $\AA$; that is, $\WW$ is closed under kernels, images and cokernels. Then we can define the complete lattice $\tors \WW$ as before and compare it with the interval $[\UU,\TT]$ (see Theorem \ref{theorem:reduction:intro}). Note that such subcategories $\WW$ are always closed under extensions, as $\UU^\perp$ and $\TT$ are. Thus, if we assume that $\WW$ is abelian, it will be automatically a {\it wide subcategory} of $\AA$, that is, an abelian subcategory closed under extensions. This motivates the following definition.

\begin{definition}[Definition \ref{definition:wideintervals}]\label{intro:definition:wideintervals}
We call an interval $[\UU,\TT]$ in $\tors \AA$ a \textit{wide interval} 
if $\UU^\perp \cap \TT$ is a wide subcategory of $\AA$.
\end{definition}

An important class of examples of wide intervals is given by numerical torsion classes 
in the case $\AA=\mod A$, where $A$ is a finite-dimensional algebra over a field $K$.
Let us write $\proj A$ for the category of finitely generated projective $A$-modules.

\begin{proposition}\cite{BKT,B2}
For each element $\theta \in K_0(\proj A)\otimes_{\integers}\reals$, define two numerical 
torsion classes $\TT_\theta \subseteq \overline{\TT}_\theta$ in $\mod A$ by 
\begin{align*}
\TT_\theta &:= \{ M \in \mod A \mid 
\text{for any nonzero quotient $N$, $\theta(N) > 0$}\}, \\
\overline{\TT}_\theta &:= \{ M \in \mod A \mid 
\text{for any quotient $N$, $\theta(N) \ge 0$}\}.
\end{align*}
Then $[\TT_\theta, \overline{\TT}_\theta]$ is a wide interval
with the associated wide subcategory 
$\WW_\theta:=(\TT_\theta)^\perp \cap \overline{\TT}_\theta$ is 
the \textit{$\theta$-semistable subcategory} defined by King \cite{K}.
\end{proposition}

See also \cite{BST,Y,A2} for more information on numerical torsion classes.

Another important example appears in \textit{$\tau$-tilting reduction} as developed by Jasso \cite{J}. 

\begin{proposition}\cite[Theorem 3.12]{J}
For a $\tau$-rigid module $U\in\mod A$, 
the interval $[\Fac U,{^\perp(\tau U)}]$ is a wide interval. 
This interval is isomorphic to the lattice 
$\tors \WW_{U}$ of torsion classes in $\WW_{U}:=U^\perp\cap{^\perp(\tau U)}$.
\end{proposition}

We aim to extend this result in this paper.

In general, we call two torsion classes $\UU \subsetneq \TT$ in $\AA$ \textit{adjacent} if there exists no torsion class $\VV$ in $\AA$ satisfying $\UU \subsetneq \VV \subsetneq \TT$. 
For example, if $[\UU,\TT]=[\Fac U, {^\perp}(\tau U)]$ and 
$U$ is an almost support $\tau$-tilting module,
then $\UU$ and $\TT$ are adjacent \cite[Example 3.5]{DIJ}.
For two adjacent torsion classes $\UU\subsetneq\TT$, the subcategory $\WW=\UU^\perp\cap\TT$ is wide and contains a unique brick $S$ \cite{BCZ,DIRRT}.
In particular, we can label the arrow $\TT \to \UU$ 
in the Hasse quiver $\Hasse(\tors \AA)$ of the partially ordered set $\tors \AA$ by the brick $S$.
This is called \textit{brick labeling} of $\Hasse(\tors \AA)$ introduced in \cite{A} for Hasse arrows between functorially finite torsion classes and in general in \cite{DIRRT}.

Our first result shows that the wide interval $[\UU,\TT]$ and $\tors \WW$ are isomorphic complete lattices and that this isomorphism is compatible with their brick labelings.

\begin{theorem}[Theorem \ref{theorem:reduction}]\label{theorem:reduction:intro}
Let $[\UU,\TT]$ be a wide interval in $\tors \AA$ and set $\WW:=\UU^\perp\cap\TT$. 
\begin{itemize}
\item[(1)]
There are mutually inverse isomorphisms of complete lattices
\begin{align*}\xymatrix@R0em@C4em{
[\UU,\TT]\ar@<0.5ex>[r]^{\Phi}&\ar@<0.5ex>[l]^{\Psi}\tors \WW
}\end{align*}
given by $\Phi(\VV):=\UU^\perp\cap\VV$ and $\Psi(\X):=\T(\UU,\X)$ for $\VV\in[\UU,\TT]$ and $\X\in\tors \WW$ respectively. 
Moreover, $\Psi(\X)=\UU*\X$ holds for every $\X\in\tors \WW$.
\item[(2)]
The isomorphism $\Phi$ preserves the brick labeling: Namely, the brick label of $\VV_1 \to \VV_2$ in $[\UU,\TT]$ is the same as the brick label of $\Phi(\VV_1) \to \Phi(\VV_2)$ in $\tors \WW$.
\item[(3)]
The following sets coincide:
\begin{itemize}
\item The set $\simple \WW$ of isomorphism classes of simple objects in $\WW$. 
\item The set of labels of arrows in the Hasse quiver of $[\UU,\TT]$ starting at $\TT$.
\item The set of labels of arrows in the Hasse quiver of $[\UU,\TT]$ ending at $\UU$.
\end{itemize}
\end{itemize}
\end{theorem}

This is a large extension of 
\cite[Theorem 3.12]{J} and 
\cite[Theorem 4.12, Proposition 4.13, Theorem 4.16]{DIRRT},
which deal with functorially finite torsion classes in $\mod A$.

Thus, it is natural to characterize wide intervals in terms of Hasse arrows in $\tors \AA$.
For this purpose, we define 
\begin{align*}
[\UU,\TT]^+&:=\{\TT\}\cup\{\VV\in[\UU,\TT] \mid \text{there exists
$(\TT\to\VV)\in\Hasse(\tors \AA)$}\},\\
[\UU,\TT]^-&:=\{\UU\}\cup\{\VV\in[\UU,\TT] \mid \text{there exists
$(\VV\to\UU)\in\Hasse(\tors \AA)$}\},
\end{align*}
and introduce the following two lattice theoretic notions for intervals in $\tors \AA$:

\begin{definition}[Definition \ref{definition:joinintervals}]
Let $[\UU,\TT]$ be an interval in $\tors \AA$.
\begin{itemize}
\item[(1)]
We call $[\UU,\TT]$ \textit{a join interval} if $\TT=\Join [\UU,\TT]^-$.
\item[(2)]
We call $[\UU,\TT]$ \textit{a meet interval} if $\UU=\Meet [\UU,\TT]^+$.
\end{itemize}
\end{definition}

Our second main result characterizes wide intervals as precisely the join and the meet intervals.

\begin{theorem}[Theorem \ref{theorem:joinintervals}]
Let $[\UU,\TT]$ be an interval in $\tors \AA$. 
The following conditions are equivalent:
\begin{itemize}
\item[(a)] The interval $[\UU,\TT]$ is a wide interval.
\item[(b)] The interval $[\UU,\TT]$ is a join interval.
\item[(c)] The interval $[\UU,\TT]$ is a meet interval.
\end{itemize}
\end{theorem}

Apart from wide intervals, there are other operations connecting torsion classes and wide subcategories.
First, the operation $\T$ given in (\ref{minimaltorsionclass}) defines a map from the set $\wide \AA$ of wide subcategories in $\AA$ to $\tors \AA$.
On the other hand, 
Ingalls--Thomas \cite{IT} and Marks--\v{S}\v{t}ov\'{i}\v{c}ek \cite{MS} 
introduced a map $\WL \colon \tors \AA \to \wide \AA$ 
(see (\ref{WL}) for the definition)
such that the composite
 $\WL \circ \T$ is the identity.
It is therefore important to understand the torsion classes in the image of the map $\T \colon \wide \AA \to \tors \AA$,
which we call \textit{widely generated torsion classes}.

For this purpose, we will first prove that the wide subcategory $\WL(\TT)$ is related to wide intervals as follows. 
Below, $\Serre(\WL(\TT))$ denotes the set of \textit{Serre subcategories} of $\WL(\TT)$.

\begin{theorem}[Theorem \ref{theorem:power}]
Let $\TT \in \tors \AA$. Taking labels gives a bijection
\begin{align*}
\{ \text{Hasse arrows in $\tors \AA$ starting at $\TT$} \} \to \simple(\WL(\TT)).
\end{align*}
Moreover, the map $\WW \mapsto \TT \cap {^\perp \WW}$ induces a bijection
\begin{align*}
\Serre(\WL(\TT)) \to \{ \UU \in \tors \AA \mid \text{$[\UU,\TT]$ is a wide interval}\}.
\end{align*}
\end{theorem}

As a consequence, we obtain the following characterization of widely generated torsion classes:

\begin{theorem}[Theorem \ref{theorem:widetorsionclass}]
For $\TT\in\tors \AA$ the following conditions are equivalent:
\begin{itemize}
\item[(a)] $\TT$ is a widely generated torsion class.
\item[(b)] $\TT=\T(\WL(\TT))$.
\item[(c)] $\TT=\T(\LL)$, where $\LL$ is the set of labels of Hasse arrows in $\tors \AA$ starting at $\TT$.
\item[(d)] For every $\UU\in\tors \AA$ with $\UU\subsetneq\TT$, there exists a Hasse arrow $\TT \to \UU'$ such that $\UU \subseteq \UU'$.
\end{itemize}
\end{theorem}


\section{Setting and notation}\label{section:setting}

In this section we recall some fundamental facts about partially ordered sets and abelian length categories and describe the standing assumptions of this paper.

\subsection{Partially ordered sets}

Let $L=(L,\le)$ be a partially ordered set.

First, we define the \textit{opposite partially ordered set} 
$L^\op:=(L^\op,\le^\op)$ of $L$
with the same underlying set $L=L^\op$ and the partial order $\le^\op$
such that $y\le^\op x$ holds
if and only if $x\le y$ for all $x,y\in L$.

An \textit{isomorphism of partially ordered sets} $\Phi \colon (L,\le)\to(L',\le')$ 
is a bijection of sets $\Phi:L\to L'$ such that 
$\Phi(y)\le\Phi(x)$ if and only if $y\le x$.

An \textit{interval} in $L$ is a subset of the form
\begin{align*}
[y,x]:=\{z\in L \mid y\le z\le x\}
\end{align*}
where $y\le x$ are elements of $L$. 

The \textit{Hasse quiver} $\Hasse L$ has $L$ as its set of vertices, and there is an arrow $x\to y$ for $x,y\in L$ if and only if $y<x$ and $[y,x]=\{y,x\}$. 

For an interval $[y,x]$ in $L$, 
we define its set of \textit{upper} (resp.~ \textit{lower}) \textit{elements} as follows:
\begin{align*}
[y,x]^+&:=\{x\}\cup\{z\in[y,x] \mid \text{there exists $(x\to z)\in\Hasse L$}\},\\
[y,x]^-&:=\{y\}\cup\{z\in[y,x] \mid \text{there exists $(z\to y)\in\Hasse L$}\}.
\end{align*}
In particular, if $y=x$, then $[y,x]^+=[y,x]^-=\{x\}$.

Let $S\subseteq L$ be a subset. A \textit{meet} (resp.~ \textit{join}) of $S$ is a largest (resp.~ smallest) element in the set $\{y \in L \mid \text{$y \le x$ for all $x \in S$}\}$
(resp.~ $\{y \in L \mid \text{$y \ge x$ for all $x \in S$}\}$).
If a meet (resp.~ join) of $S$ exists, then it is necessarily unique, 
and we will refer to it as the meet (resp.~ the join) of $S$ denoted by 
\begin{align*}
\Meet S=\Meet_{x\in S}x 
\quad 
\left(\text{resp.~} \Join S=\Join_{x\in S}x\right).
\end{align*}
If a meet and a join of every subset of $L$ exists, we call $L$ a \textit{complete lattice}.

\subsection{Abelian length categories}

In this paper, every category is assumed to be essentially small, 
and every subcategory is a full subcategory closed under isomorphism classes. 
Throughout $\AA$ always denotes an \textit{abelian length category}, 
that is, every object $X \in \AA$ 
has a finite filtration 
$0 = X_0 \subsetneq X_1 \subsetneq \cdots \subsetneq X_l = X$
where each subfactor $X_{i+1}/X_i$ a simple object in $\AA$.

The \textit{opposite category} $\AA^\op$ of $\AA$ is 
defined to be the category with the same objects as $\AA$,
but $\Hom_{\AA^\op}(X,Y):=\Hom_{\AA}(Y,X)$ for all $X,Y\in\AA$. 
Note that $\AA^\op$ is again an essentially small abelian length category if $\AA$ is so.

For a subcategory $\CC\subseteq\AA$ define the following subcategories of $\AA$:
\begin{align*}
\add \CC &:=\{X \in \AA \mid \text{there exist $C \in\CC$, $n\in\naturals$ and a split epimorphism $C^n\twoheadrightarrow X$}\}, \\
\Filt \CC &:=\{X \in \AA \mid \text{there exists $0=X_0 \subseteq X_1 \subseteq \cdots \subseteq X_n=X$ with $X_i/X_{i-1} \in \add \CC$}\}, \\
\Fac \CC &:=\{X \in \AA \mid 
\text{there exist $C \in \add \CC$ and an epimorphism 
$C\twoheadrightarrow X$ in $\AA$} \}, \\
\Sub \CC &:=\{X \in \AA \mid 
\text{there exist $C\in\add \CC$ and a monomorphism 
$X\hookrightarrow C$ in $\AA$}\}, \\
\CC^\perp &:= \{X \in \AA \mid 
\text{for all $C\in\CC$,  $\Hom_\AA(C,X)=0$}\}, \\
{^\perp}\CC &:= \{X \in \AA \mid 
\text{for all $C\in\CC$,  $\Hom_\AA(X,C)=0$}\}.
\end{align*}
We will use abbreviation rules such as
\begin{align*}
\add (\CC_1,\ldots,\CC_m,X_1,\ldots,X_n):=\add (\CC_1\cup\cdots\cup\CC_m\cup\{X_1,\ldots,X_n\})
\end{align*}
for $\CC_1,\ldots,\CC_m \subseteq \AA$ and $X_1,\ldots,X_n \in \AA$.
Also, if $\CC,\CC' \subseteq \AA$, then we write $\CC * \CC'$ for the full subcategory 
of objects $X$ admitting a short exact sequence
\[
0 \to C \to X \to C' \to 0
\] 
with $C \in \CC$ and $C' \in \CC'$.

A subcategory $\WW\subseteq\AA$ is called \textit{wide} if $\WW$ is closed under kernels, cokernels and extensions in $\AA$. In particular, every wide subcategory of $\AA$ is an abelian subcategory. We denote by $\simple \WW$ the set of isomorphism classes of simple objects in $\WW$. By abuse of notation we will not always distinguish between isomorphism classes in $\simple \WW$ and their representatives. Write $\wide \AA$ for the set of wide subcatgeories of $\AA$.

A \textit{brick} in $\AA$ is an object $S\in\AA$ such that $\End_A(S)$ is a division ring. For a subcategory $\CC\subseteq\AA$, let $\brick \CC$ be the set of isomorphism classes of bricks in $\CC$.  A set $\S$ of isomorphism classes of bricks in $\AA$ is called a \textit{semibrick} if its elements are pairwise $\Hom$-orthogonal, that is, $\Hom_A(S,S')=0$ for any two representatives $S$ and $S'$ of distinct isomorphism classes in $\S$. Let $\sbrick \AA$ be the set of semibricks in $\AA$.

A classical result of Ringel \cite[1.2]{R} tells us that $\Filt$ and $\simple$ induce mutually inverse bijections
\[
\xymatrix@C4em{
\sbrick \AA\ar@<0.5ex>[r]^{\Filt}&\ar@<0.5ex>[l]^{\simple}\wide \AA.
}
\]

A subcategory $\WW$ of $\AA$ is called
a \textit{Serre subcategory} if it is closed under extensions, factor objects and subobjects. Since $\AA$ is an abelian length category, a subcategory $\WW\subseteq\AA$ is a Serre subcategory if and only if $\WW=\Filt \S$ for a subset $\S\subseteq\simple\AA$.
We denote the set of Serre subcategories of $\AA$ by $\Serre \AA$.
This set can be identified with the power set of $\simple \AA$.

A \textit{torsion pair} $(\TT,\FF)$ in $\AA$ is a pair of subcategories $\TT,\FF\subseteq\AA$ such that 
\begin{align*}
\FF=\TT^\perp \text{ and } \TT={^\perp\FF}.
\end{align*}
For a torsion pair $(\TT,\FF)$ and every object $X\in\AA$,
there exists a short exact sequence
\begin{align}\label{canonicalsequence}
0\to t_{\TT}X\to X\to f_{\FF}X\to 0
\end{align}
with $t_{\TT}X\in\TT$ and $f_{\FF}X\in\FF$, unique up to isomorphism. We will refer to (\ref{canonicalsequence}) as \textit{the canonical sequence} of $X$ with respect to the torsion pair $(\TT,\FF)$. Note that the subcategory $\TT$ in a torsion pair $(\TT,\FF)$ is closed under extensions and taking factor objects. Dually, the subcategory $\FF$ is closed under extensions and subobjects. Conversely, one can show that, given a subcategory $\TT\subseteq\AA$ closed under extensions and factor objects, the pair $(\TT,\TT^\perp)$ is a torsion pair, and we call such subcategories $\TT$ \textit{torsion classes}. The set of torsion classes in $\AA$ is denoted by $\tors \AA$. Dually, if $\FF\subseteq\AA$ is a subcategory closed under extensions and subobjects, then $({^\perp\FF},\FF)$ is a torsion pair, and we call $\FF$ a torsion-free class. The set of torsion-free classes in $\AA$ is denoted by $\torf \AA$.

The sets $\tors \AA$ and $\torf \AA$ are partially ordered by inclusion and closed under arbitrary intersections. In particular, given a full subcategory $\X\subseteq\AA$,
\begin{align*}
\T(\X):=\bigcap_{\TT\in\tors \AA; \X\subseteq\TT}\TT \quad \text{and} \quad \F(\X):=\bigcap_{\FF\in\torf \AA; \X\subseteq\FF}\FF
\end{align*}
are the smallest torsion class and the smallest torsion-free class containing $\X$ respectively. It is well known that $\T=\Filt \circ \Fac$ and $\F=\Filt \circ \Sub$, see \cite[Lemma 3.1]{MS} for a proof. These operators show that $\tors \AA$ and $\torf \AA$ are complete lattices with joins and meets given by
\begin{align*}
\Join S&=\Join_{\TT\in S}\TT=\T\left(\bigcup_{\TT \in S} \TT\right), & 
\Meet S&=\Meet_{\TT\in S}\TT=\bigcap_{\TT \in S} \TT, \\
\Join S'&=\Join_{\FF\in S'}\FF=\F\left(\bigcup_{\FF \in S'} \FF\right), & 
\Meet S'&=\Meet_{\FF\in S'}\FF=\bigcap_{\FF \in S'} \FF
\end{align*}
for $S \subseteq \tors \AA$ and $S' \subseteq \torf \AA$.
We remark that $(\T(\X),\X^\perp)$ and $({^\perp \X},\F(\X))$ are torsion pairs in $\AA$ for any full subcategory $\X \subseteq \AA$.

In examples, we sometimes let $A$ be a finite-dimensional algebra over a field $K$.
Then we let $\AA=\mod A$ be the category of finite-dimensional $A$-modules,
and write $\tors A:=\tors(\mod A)$.
We write $\proj A$ for the category of finitely generated projective $A$-modules.

\section{Brick labeling}\label{section:bricklabeling}

Let $[\UU,\TT]$ be an interval in $\tors \AA$. 
The subcategory $\UU^\perp\cap\TT$ measures the ``difference'' between $\UU$ and $\TT$ in the following sense:

\begin{lemma}\label{lemma:intervalcategory}
Let $[\UU,\TT]$ be an interval in $\tors \AA$ and set $\WW:=\UU^\perp\cap\TT$. Then for every $T\in\TT$ there is a short exact sequence
\begin{align*}
0\to U\to T\to W\to 0
\end{align*}
with $U\in\UU$ and $W\in\WW$, unique up to isomorphism of short exact sequences.
Thus $\TT=\UU*\WW$.
\end{lemma}
\begin{proof}
Take the canonical sequence (\ref{canonicalsequence}) for $T$ with respect to the torsion pair $(\UU,\UU^\perp)$ and let $U:=t_{\UU}T\in\UU$ and $W:=f_{\UU^\perp}T\in\UU^\perp$. 
Since $\TT$ is closed under factor objects, we obtain $W\in\TT$; 
hence $W\in\UU^\perp\cap\TT=\WW$.
Conversely, the desired exact sequence is unique up to isomorphism 
by the uniqueness of the canonical sequence for $T$.
Thus we obtain the first statement, which readily implies that $\TT=\UU*\WW$.
\end{proof}

In \cite{DIRRT}, the following fundamental properties of intervals in $\tors \AA$ are established.

\begin{proposition}\label{proposition:interval}
Let $[\UU,\TT]$ be an interval in $\tors \AA$. 
\begin{itemize}
\item[(1)]
\cite[Lemma 3.10]{DIRRT}
Then $\UU^\perp\cap\TT=\Filt(\brick(\UU^\perp\cap\TT))$.
\item[(2)]
\cite[Theorems 3.3, 3.4]{DIRRT}
There is an arrow $q \colon \TT\to\UU$ in $\Hasse(\tors \AA)$ if and only if there exists a unique brick $S_q \in \brick(\UU^\perp\cap\TT)$ up to isomorphism. In this case,
\begin{align*}
\UU^\perp\cap\TT=\Filt S_q
\end{align*}
and, moreover, $\UU=\TT \cap {^\perp S_q}$ and $\TT=\T(\UU,S_q)$.
\end{itemize}
\end{proposition}

%

Proposition \ref{proposition:interval} gives rise to the following brick labeling.

\begin{definition}[{\cite[Definition 3.5]{DIRRT}}]\label{definition:bricklabeling}
The \textit{brick label} of an arrow $q \colon \TT \to \UU$ in $\Hasse(\tors \AA)$ is the isomorphism class of the brick $S_q$ from Proposition \ref{proposition:interval}.
In this case, we write $\TT \xto{S_q} \UU$. 
\end{definition}

The brick labeling may be compared to minimal extending modules for $\UU$ as defined in Barnard--Carroll--Zhu \cite{BCZ}. The minimal extending modules for $\UU$ are precisely the labels of arrows ending in $\UU$,
see \cite[Subsection 2.2]{BCZ}.

Dually, define brick labeling of $\Hasse(\torf \AA)$ so that
an arrow $\FF \to \GG$ is labeled by the unique brick in
${^\perp} \GG \cap \FF$. There is the following strong relationship between $\tors \AA$ and $\torf \AA$ compatible with their brick labelings.

\begin{proposition}\label{proposition:duality}
The construction of right and left $\Hom$-orthogonal subcategories defines isomorphisms of partially ordered sets
\begin{align*}
\xymatrix@R1.2em@C4em{
(\tors \AA)^\op\ar@<0.5ex>[r]^{{\bullet}^\perp}&\ar@<0.5ex>[l]^{{^\perp\bullet}}\torf \AA.
}
\end{align*}
Moreover, these isomorphisms preserve the brick labeling in the following sense: The brick label of $\TT\to\UU$ in $\Hasse(\tors \AA)$ is the same as the brick label of $\UU^\perp\to\TT^\perp$ in $\Hasse(\torf \AA)$.
\end{proposition}
\begin{proof}
The isomorphism $(\tors \AA)^\op\cong \torf \AA$ is well known and elementary. We show the invariance of the brick labeling. Let $\TT\xto{S_q}\UU$ be a labeled arrow in $\Hasse(\tors \AA)$. Note that ${^\perp(\TT^\perp)}=\TT$ since $\TT$ is a torsion class in $\AA$. This basic observation implies
\begin{align*}
{^\perp(\TT^\perp)}\cap\UU^\perp=\UU^\perp\cap\TT=\Filt S_q;
\end{align*}
hence, according to the dual of Definition \ref{definition:bricklabeling}, $\UU^\perp\to\TT^\perp$ has the label $S_q$ in $\Hasse(\torf \AA)$.
\end{proof}

Brick labeling between torsion classes was first considered for the preprojective algebras $\Pi$ of Dynkin type $\Delta$ in \cite{IRRT}. The crucial ingredient is the bijection between the Coxeter group $W$ associated to $\Delta$ and $\tors \Pi$ established by Mizuno \cite{M}. The first named author of this paper introduced brick labeling in \cite{A} for functorially finite torsion classes in the module category $\mod A$ 
in the context of the Koenig--Yang correspondences \cite{KY,BY} and 
$\tau$-tilting theory \cite{AIR}.

One can easily check that the labeled Hasse quiver of $\tors \AA$ has the following global structure:

\begin{proposition}\label{proposition:endpoints}
For an abelian length category $\AA$, the following properties hold:
\begin{itemize}
\item The arrows in $\Hasse(\tors \AA)$ ending in $0$ are 
$\Filt S \xto{S}0$ where $S$ runs through $\simple \AA$.
\item The arrows in $\Hasse(\tors \AA)$ starting at $\AA$ are 
$\AA\xto{S}{^\perp S}$ where $S$ runs through $\simple \AA$.
\item For $S,S'\in\simple \AA$, there is an inclusion $\Filt S \subseteq{^\perp S'}$ if and only if $S\not\cong S'$.
\end{itemize} 
\end{proposition}

Proposition \ref{proposition:incidentsemibrick} below shows that the set of labels of arrows in $\Hasse(\tors \AA)$ starting at (resp.~ ending in) a fixed torsion class actually form a semibrick. In order to simplify our statements, we introduce the following notation.

\begin{notation}\label{notation:labels}
For a subset $S\subseteq\tors \AA$, let $\Label S$ denote the set of labels of arrows in the full subquiver of $\Hasse(\tors \AA)$ with vertices in $S$.
\end{notation}

\begin{proposition}\label{proposition:incidentsemibrick} 
Let $[\UU,\TT]$ be an interval in $\tors \AA$. The sets $\Label{[\UU,\TT]^+}$ and $\Label{[\UU,\TT]^-}$ are semibricks.
\end{proposition}

\begin{proof}
We prove only the first statement; the second one can be shown similarly.

Since all labels are bricks by Definition \ref{definition:bricklabeling}, it is enough to show $\Hom$-orthogonality.
Let $\TT_i\xto{S_i}\TT$ be arrows in $\Hasse(\tors \AA)$ for $i\in\{1,2\}$
with $\TT_1 \ne \TT_2$.
We first remark that $S_1 \not \cong S_2$ follows from Proposition \ref{proposition:interval}.

Let $f\in\Hom_{\AA}(S_1,S_2)$.
There is a short exact sequence
\begin{align*}
0 \to X' \to \Cokernel f \to X'' \to 0
\end{align*}
with $X' \in \TT$ and $X'' \in \TT^\perp$.
Since $X''$ is a factor object of $S_2$, we deduce that $X'' \in \TT^\perp \cap \TT_2$.
Then Proposition \ref{proposition:interval} implies $X'' \in \Filt S_2$.
Thus $X''$ must be $0$ or isomorphic to $S_2$.

If $X''=0$, then $\Cokernel f \in \TT$.
Thus in the short exact sequence 
\begin{align*}
0 \to \Image f \to S_2 \to \Cokernel f \to 0,
\end{align*}
$\Image f \in \T(S_1) \subseteq \TT_1$ and $\Cokernel f \in \TT \subseteq \TT_1$ hold, so $S_2 \in \TT_1$.
Clearly, $S_2 \in \TT^\perp$, so $S_2 \in \TT^\perp \cap \TT_1$.
Then Proposition \ref{proposition:interval} gives $S_2 \cong S_1$,
but this is a contradiction.
Thus $X'' \cong S_2$, and then $\Cokernel f = S_2$.
This implies $f=0$ as desired.
\end{proof}

\section{Wide intervals and reduction of torsion classes}\label{section:reduction}

In this section we investigate wide intervals $[\UU,\TT]$ in $\tors \AA$.

\begin{definition}\label{definition:wideintervals}
An interval $[\UU,\TT]$ in $\tors \AA$ is a \textit{wide interval} 
if $\UU^\perp\cap\TT$ is a wide subcategory of $\AA$.
\end{definition}

For example, 
we can construct a wide interval $[\UU(N,Q),\TT(N,Q)] \subseteq \tors A$
from a $\tau$-rigid pair $(N,Q)$ in $\mod A$
(that is, $N \in \mod A$ and $Q \in \proj A$ satisfying
$\Hom_A(N,\tau N)=0$ and $\Hom_A(Q,N)=0$), where 
\[\UU(N,Q):=\Fac N \text{ and } \TT(N,Q):={^\perp (\tau N)} \cap Q^\perp\]
as in \cite{J,DIRRT}.
In this case, $[\UU(N,Q),\TT(N,Q)]$ is isomorphic to $\tors C_{N,Q}$
for a certain finite-dimensional $K$-algebra $C_{N,Q}$
\cite[Theorem 4.12]{DIRRT} (see also \cite[Theorems 3.8, 3.12]{J}) as a complete lattice. This bijection is compatible with the brick labeling of 
$[\UU(N,Q),\TT(N,Q)] \subseteq \tors A$ and $\tors C_{N,Q}$ 
\cite[Proposition 4.13]{DIRRT}.
We extend these results for all wide intervals as follows.

\begin{theorem}\label{theorem:reduction}
Let $[\UU,\TT]$ be a wide interval in $\tors \AA$, and set $\WW:=\UU^\perp\cap\TT$. 
\begin{itemize}
\item[(1)]
There are mutually inverse isomorphisms of complete lattices
\begin{align*}\xymatrix@R0em@C4em{
[\UU,\TT]\ar@<0.5ex>[r]^{\Phi}&\ar@<0.5ex>[l]^{\Psi}\tors \WW
}\end{align*}
given by $\Phi(\VV):=\UU^\perp\cap\VV$ and $\Psi(\X):=\T(\UU,\X)$ for $\VV\in[\UU,\TT]$ and $\X\in\tors \WW$ respectively. 
Moreover, $\Psi(\X)=\UU*\X$ holds for every $\X\in\tors \WW$.
\item[(2)]
The isomorphism $\Phi$ preserves the brick labeling:
Namely, the brick label of $\VV_1 \to \VV_2$ in $[\UU,\TT]$ is the same as the brick label of $\Phi(\VV_1) \to \Phi(\VV_2)$ in $\tors \WW$.
\item[(3)]
We have $\simple \WW=\Label{[\UU,\TT]^+}=\Label{[\UU,\TT]^-}$.
\end{itemize}
\end{theorem}
\begin{proof} 
(1) 
Both maps are easily seen to be well defined morphisms of partially ordered sets.

For $\VV\in[\UU,\TT]$, the identity $\VV=\UU*(\UU^\perp\cap\VV)$ follows immediately
from Lemma \ref{lemma:intervalcategory}, 
but $\UU*(\UU^\perp\cap\VV)\subseteq\T(\UU,\UU^\perp\cap\VV)$ hence $\VV=\T(\UU,\UU^\perp\cap\VV)$ by minimality.
This proves $\Psi\circ\Phi=\id$.

Let $\X\in\tors \WW$. 
Since $\X \subseteq \WW \subseteq \UU^\perp$, 
the inclusion $\X \subseteq \UU^\perp \cap \T(\UU,\X)$ is obvious. 
For the other one, let $Y \in \UU^\perp \cap \T(\UU,\X)$.
We use induction on the length of $Y$ in $\WW$.
If $Y=0$, then the claim is clear.
Assume $Y \ne 0$; then there is a short exact sequence
\begin{align*}
0\xto{}Y'\xto{}Y\xto{}Y''\xto{}0
\end{align*}
with $0\neq Y'\in \T(\X) \subseteq \AA$ and $Y''\in \X^\perp$.
Then $Y' \in \UU^\perp$ because $Y'$ is a subobject of $Y$,
and $Y' \in \T(\X) \subseteq \TT$ follows from $\X \subseteq \WW \subseteq \TT$.
Thus $Y'\in \UU^\perp \cap \TT \cap \T(\X) = \WW \cap \T(\X)=\X$, 
since $\X \in \tors \WW$.
Moreover $Y'' \cong \Cokernel (Y' \to Y) \in \WW$ as $\WW$ is assumed to be wide,
and we have $Y'' \in \UU^\perp \cap \T(\UU,\X)$ since
$\WW \subseteq \UU^\perp$ and $Y''$ is a factor object of $Y \in \T(\UU,\X)$. 
Now we may apply induction on the length to $Y''$.
We obtain $Y''\in\X$ and henceforth $Y\in\X$. This proves $\Phi\circ\Psi=\id$.

Altogether $\Phi$ and $\Psi$ are mutually inverse isomorphisms of partially ordered sets thus also isomorphisms of complete lattices.

For the last statement, let $\X \in \tors \WW$,
then we get 
\[\Psi(\X)=\UU*(\UU^\perp\cap\Psi(\X))=\UU*\Phi(\Psi(\X))=\UU*\X.\]

(2) For every interval $[\VV_1,\VV_2]$ in $[\UU,\TT]$ the equality
\begin{align*}
(\UU^\perp\cap\VV_1)^\perp\cap(\UU^\perp\cap\VV_2)=\VV_1^\perp\cap\VV_2\subseteq\WW
\end{align*}
holds, hence $\Phi$ preserves the brick labeling by Proposition \ref{proposition:interval},
and so does $\Psi$.

(3) In view of Proposition \ref{proposition:endpoints}, this follows from parts (1) and (2).
\end{proof}

Let us give an example of a wide interval which does not come from $\tau$-rigid pairs.

\begin{example}\label{example:Kronecker}
Assume that $K$ is an algebraically closed field and that $A$ is the path algebra
$K(1 \rightrightarrows 2)$ of the Kronecker quiver $1 \rightrightarrows 2$.

We set $\TT$ as the smallest torsion class containing all regular and preinjective modules,
and $\UU$ as the smallest torsion class containing all preinjective modules.

In this case, $\WW$ is a wide subcategory of $\mod A$, namely the subcategory of all regular $A$-modules. Thus $[\UU,\TT]$ is a wide interval. The simple objects in $\WW$ are all quasi-simple regular modules $S_\lambda$ parametrized by 
$\lambda \in \mathbb{P}^1(K)$.
One can check that $\WW \simeq \bigoplus_{\lambda \in \mathbb{P}^1(K)} \Filt S_\lambda$
as an abelian category, so $\tors \WW$ can be identified with 
$\prod_{\lambda \in \mathbb{P}^1(K)} \tors(\Filt S_\lambda)$.
Since the torsion classes in $\Filt S_\lambda$ are $\Filt S_\lambda$ and $\{0\}$,
$\tors \WW$ is in bijection with the power set $2^{\mathbb{P}^1(K)}$.

Therefore Theorem \ref{theorem:reduction} gives an isomorphism of complete lattices
\begin{align*}
2^{\mathbb{P}^1(K)} \to [\UU,\TT]; \qquad \Lambda \mapsto \UU*\Filt \{S_\lambda \mid \lambda \in \Lambda\}.
\end{align*}
Every arrow ending in the smallest element $\{0\} \in \tors \WW$ is of the form $\Filt S_\lambda \to \{0\}$ for some $\lambda \in \mathbb{P}^1(K)$,
and it is labeled by the brick $S_\lambda$.
The corresponding arrow in $[\UU,\TT]$ is $\T(\UU,S_\lambda) \to \UU$,
and its brick label is also $S_\lambda$.

We remark that this example can be obtained also from numerical torsion classes
\begin{align*}
\TT_\theta &:= \{ M \in \mod A \mid 
\text{for any nonzero quotient $N$, $\theta(N) > 0$}\}, \\
\overline{\TT}_\theta &:= \{ M \in \mod A \mid 
\text{for any quotient $N$, $\theta(N) \ge 0$}\}.
\end{align*}
associated to each $\theta \in K_0(\proj A) \otimes_\integers \reals$ in \cite{BKT,B2}.
For any $\theta$, the intersection
$\TT_\theta^\perp \cap \overline{\TT}_\theta$ is the $\theta$-semistable subcategory
\begin{align*}
\WW_\theta:=\{ M \in \mod A \mid 
\text{$\theta(M)=0$, and for any nonzero quotient $N$, $\theta(N) \ge 0$}\}
\end{align*}
introduced by King \cite{K}, which is a wide subcategory.
By setting $\theta:=[P_1]-[P_2]$ with $P_i$ the indecomposable projective module,
we get $\TT_\theta=\UU$ and $\overline{\TT}_\theta=\TT$ above.
Thus, $[\UU,\TT]$ is a wide interval, and the simple objects of $\WW_\theta$ are
$S_\lambda$.

Similar arguments hold for tame hereditary algebras.
\end{example}

\section{Classification of wide intervals in terms of join and meet intervals}\label{section:joinmeetintervals}

Motivated by the results of Section \ref{section:reduction}, we next aim for characterizing
wide intervals in terms of arrows in the Hasse quiver of torsion classes.
For this purpose, we define the following notions for intervals.

\begin{definition}\label{definition:joinintervals}
Let $[\UU,\TT]$ be an interval in $\tors \AA$.
\begin{itemize}
\item[(1)]
The interval $[\UU,\TT]$ is called a \textit{join interval} if $\TT=\Join [\UU,\TT]^-$.
\item[(2)]
The interval $[\UU,\TT]$ is called a \textit{meet interval} if $\UU=\Meet [\UU,\TT]^+$.
\end{itemize}
\end{definition}

We remark that $[\TT,\TT]$ is a join interval and a meet interval
because $[\TT,\TT]^-=[\TT,\TT]^+=\{\TT\}$.
Actually, these notions coincide with wide intervals.

\begin{theorem}\label{theorem:joinintervals}
Let $[\UU,\TT]$ be an interval in $\tors \AA$. 
Then the following conditions are equivalent:
\begin{itemize}
\item[(a)] The interval $[\UU,\TT]$ is a wide interval.
\item[(b)] The interval $[\UU,\TT]$ is a join interval.
\item[(c)] The interval $[\UU,\TT]$ is a meet interval.
\end{itemize}
\end{theorem}

The isomorphisms between the complete lattices in Theorem \ref{theorem:reduction} together with Proposition \ref{proposition:endpoints} imply $\text{(a)} \Rightarrow \text{(b)}$ and $\text{(a)} \Rightarrow \text{(c)}$.
For the remaining parts, we need the following property of join intervals.

\begin{proposition}\label{proposition:joinintervals}
Let $[\UU,\TT]$ be a join interval in $\tors \AA$, set $\WW:=\UU^\perp\cap\TT$ and $\LL:=\Label{[\UU,\TT]^-}$. Then $\WW=\Filt \LL$.
\end{proposition}
\begin{proof}
The inclusion $\Filt \LL \subseteq \WW$ is immediate from $\WW$ being closed under extensions and containing $\LL$. It remains to show the other inclusion.

Proposition \ref{proposition:interval} implies that 
taking the label $S_q \in \LL$ of the arrow $q \colon \VV \to \UU$ 
for $\VV \in [\UU,\TT]^- \setminus \{\UU\}$ 
gives a bijection $[\UU,\TT]^-\setminus\{\UU\}\to\LL$.
Since $[\UU,\TT]$ is assumed to be a join interval,
\begin{align*}
\TT=\UU \join \left(\Join_{S\in\LL}\T(\UU,S)\right)=\T(\UU,\LL)=\Filt(\Fac(\UU,\LL)).
\end{align*}

We first prove that $\TT=\Filt(\UU,\LL)$. 
It is sufficient to show $X \in \Filt(\UU,\LL)$ if $X$ is a factor object of $S\in\LL$.
Set $\UU':=\T(\UU,S)$. 
We can take a short exact sequence $0 \to X' \to X \to X'' \to 0$
with $X' \in \UU$ and $X'' \in \UU^\perp$.
Then $X'' \in \UU^\perp \cap \UU'$, so $X'' \in \Filt S$ by Proposition \ref{proposition:interval}.
Thus $X \in \Filt(\UU,\LL)$, so $\TT=\Filt(\UU,\LL)$.

To finish the proof, assume $T\in\WW$ and show $T \in \Filt \LL$ 
by induction on the length of $T$ in $\WW$. 
If $T=0$, it is clear.
If $T \ne 0$, then we have a short exact sequence
$0 \to S \to T \to T' \to 0$ with $S \in \LL$,
since $T \in \TT=\Filt(\UU,\LL)$ and $T \in \WW \subseteq \UU^\perp$.
We obtain the following commutative diagram with exact rows and columns:
\begin{align*}\xymatrix@R2em@C2em{
&&0\ar[d]&0\ar[d]\\
0\ar[r]&S\ar@{=}[d]\ar[r]&P\ar[d]\ar[r]&U\ar[d]\ar[r]&0\\
0\ar[r]&S\ar[r]&T\ar[d]\ar[r]&T'\ar[d]\ar[r]&0\\
&&T''\ar[d]\ar@{=}[r]&T''\ar[d]\\
&&0&0
}\end{align*}
where $U\in\UU$ and $T''\in\UU^\perp$. 
The column $0 \to U \to T' \to T'' \to 0$ is the canonical sequence of $T'$ with respect to $(\UU,\UU^\perp)$, 
and the row $0 \to S \to P \to U \to 0$ is obtained via pullback. 
Note that $T''\in\TT$ as a proper factor of $T$, 
hence we can apply the induction hypothesis to $T'' \in \UU^\perp \cap \TT=\WW$ 
to obtain $T''\in\Filt\LL$. 
Moreover, $P\in\UU^\perp$ as a subobject of $T$ and $P\in\Filt(\UU,S)$. Thus $P\in\Filt S$ by Proposition \ref{proposition:interval}. 
Therefore $T\in\Filt\LL$ as an extension of $T''$ by $P$.
\end{proof}

Now we can complete the proof of Theorem \ref{theorem:joinintervals}.

\begin{proof}[Proof of Theorem \ref{theorem:joinintervals}]
Set $\WW:=\UU^\perp\cap\TT$ as in the statement of the theorem.

$\text{(a)} \Rightarrow \text{(b)}$, $\text{(a)} \Rightarrow \text{(c)}$: 
These follow from Theorem \ref{theorem:reduction}.

$\text{(b)} \Rightarrow \text{(a)}$:
Define $\LL:=\Label{[\UU,\TT]^-}$. 
Proposition \ref{proposition:joinintervals} says that $\WW=\Filt\LL$, and $\LL$ is a semibrick by Proposition \ref{proposition:incidentsemibrick}. Therefore $\WW$ is a wide subcategory of $\AA$ (see \cite[1.2]{R}).

$\text{(c)} \Rightarrow \text{(a)}$: 
Define $\LL:=\Label{[\UU,\TT]^+}$. 
If $[\UU,\TT]$ is a meet interval, then $[\TT^\perp,\UU^\perp]$ is a join interval in $\torf \AA$,
and $\Label{[\TT^\perp,\UU^\perp]^-}=\Label{[\UU,\TT]^+}=\LL$
is a semibrick by Proposition \ref{proposition:incidentsemibrick}.
Thus, similarly to Proposition \ref{proposition:joinintervals},
we have 
\[{^\perp(\TT^\perp)}\cap\UU^\perp = \Filt \LL \in \wide \AA.\] 
Therefore
$\WW=\UU^\perp\cap\TT={^\perp(\TT^\perp)}\cap\UU^\perp\in\wide \AA$.
\end{proof}

We end this section by giving an example illustrating Theorem \ref{theorem:joinintervals}.

\begin{example}
We keep the setting of Example \ref{example:Kronecker}.
Note 
\[ [\UU,\TT]^-=\{ \T(\UU,S_\lambda) \mid \lambda \in \mathbb{P}^1(K) \}.\]
One can check that $\TT$ coincides with the join $\Join_{\lambda \in \mathbb{P}^1(K)} \T(\UU,S_\lambda)$, which means that $[\UU,\TT]$ is a join interval,
and also a wide interval by Theorem \ref{theorem:joinintervals}. 
\end{example}

\section{Classification of wide intervals in terms of Ingalls--Thomas correspondences}\label{section:associatedwidesubcategory}

In \cite{IT} Ingalls--Thomas associate to every $\TT\in\tors \AA$ the so called 
\textit{left wide subcategory} 
\begin{align}\label{WL}
\WL(\TT):=\{X\in\TT \mid \text{for all $Y \in \TT$ and $g \colon Y \to X$, 
$\Kernel g \in \TT$} \}
\end{align}
in the case that $\AA=\mod A$ for a hereditary algebra $A$; Marks--\v{S}\v{t}ov\'{i}\v{c}ek studied $\WL(\TT)$ in \cite{MS} 
for arbitrary finite-dimensional algebras. 
Dually to the left wide subcategories associated to torsion classes, 
one can define the \textit{right wide subcategory} associated to $\FF\in\torf \AA$ 
as
\begin{align*}
\WR(\FF):=\{X\in\FF \mid \text{for all $X \in \FF$ and $f \colon X \to Y$, 
$\Cokernel f \in \FF$}\}.
\end{align*}

Their following result implies that there exists an injection from $\wide \AA$ to $\tors \AA$.
We remark that the proof in \cite{MS} also works for abelian length categories.

\begin{proposition}[{\cite[Proposition 2.14]{IT}}, {\cite[Proposition 3.3]{MS}}]\label{proposition:MS}
The maps
\begin{align*}\xymatrix@R1.2em@C4em{
\wide \AA\ar@<0.5ex>[r]^{\T}&\ar@<0.5ex>[l]^{\WL}\tors \AA
}\end{align*}
satisfy $\WL\circ\T=\id$. 
\end{proposition}

In general, $\T \colon \wide \AA \to \tors \AA$ is not surjective,
so it is important to determine the image of the map $\T$.
We will answer this problem in the next subsection,
and for this purpose, we study the relationship between $\TT \in \tors \AA$ and 
the left wide subcategory $\WL(\TT)$ in this section.

First, we prepare the following property on Serre subcategories of $\WL(\TT)$.

\begin{lemma}\label{lemma:leftwidesubcat}
Let $\TT\in\tors \AA$, $\WW\in\Serre(\WL(\TT))$ and 
$f \colon X\xto{}Y$ a homomorphism with $X\in\TT$ and $Y\in\F(\WW)$. 
Then $\Image f \in\WW$ and $\Kernel f \in\TT$.
\end{lemma}

\begin{proof}
We may assume that $f$ is surjective by replacing $Y$ by $\Image f$. We proceed by induction on the length of $Y$.

If $Y=0$, then the claim is obvious, so assume $Y \ne 0$.
Since $Y\in\F(\WW)=\Filt(\Sub\WW)$, there exists a short exact sequence
\begin{align*}
0\xto{}Y''\xto{}Y\xto{}Y'\xto{}0
\end{align*}
with $Y'\in\Sub\WW$ non-zero and $Y''\in\F(\WW)$. 
We first show $Y' \in \WW$.
Since $Y' \in \Sub \WW$, $Y'$ is a subobject of some $W \in \WW \subseteq \WL(\TT)$.
On the other hand we find $Y'\in\TT$, because $Y'$ is a factor object of $Y \in \TT$.
These two statements imply that $Y' \in \WL(\TT)$ by checking the definition of $\WL(\TT)$.
In the wide subcategory $\WL(\TT)$, $Y'$ is a subobject of some $W \in \WW$,
so we obtain $Y'\in\WW$, since $\WW\in\Serre(\WL(\TT))$.

Next we prove $Y'' \in \WW$.
Consider the following commutative diagram with exact rows:
\begin{align*}\xymatrix@R2em@C2em{
0\ar[r]&\Kernel f \ar[r]\ar[d]&X\ar[r]^f\ar@{=}[d]&Y\ar[d]\\
0\ar[r]&\Kernel f' \ar[r]&X\ar[r]^{f'}&Y'.
}\end{align*}
Again, the definition of $\WL(\TT)$ yields $\Kernel f' \in\TT$. 
Since $Y''$ is a subobject of $Y \in \F(\WW)$, we have $Y'' \in \F(\WW)$.
Now 
\begin{align*}
0\xto{}\Kernel f \xto{}\Kernel f' \xto{}Y''\xto{}0
\end{align*}
is exact and, by applying the induction hypothesis to 
$\Kernel f' \to Y''$, by induction we conclude that $Y''\in\WW$. 
Therefore $Y\in\WW$ as an extension of $Y''$ by $Y'$.
Now $\Kernel f \in \TT$ follows from the definition of $\WL(\TT)$.
\end{proof}

We can generalize Proposition \ref{proposition:interval} as follows.
This is a mutation of a torsion class $\TT$ at a Serre subcategory of $\WL(\TT)$.

\begin{proposition}\label{proposition:serremutations}
Let $\TT\in\tors \AA$, $\WW\in\Serre(\WL(\TT))$ and set $\UU:=\TT\cap{^\perp\WW}\in\tors\AA$. Then $\TT=\UU*\WW$ and $\WW=\UU^\perp\cap\TT$.
\end{proposition}
\begin{proof}
The inclusion $\UU*\WW\subseteq\TT$ is obvious from $\UU,\WW\subseteq\TT$.

For the other inclusion, let $T\in \TT$ and take the canonical sequence 
\begin{align*}
0\xto{}X\xto{}T\xto{}Y\xto{}0
\end{align*} 
with $X\in{^\perp\WW}$ and $Y\in\F(\WW)$. Lemma \ref{lemma:leftwidesubcat} implies $Y\in\WW$ and $X\in\TT\cap{^\perp\WW}=\UU$, hence $T\in\UU*\WW$.

Next we show $\WW=\UU^\perp \cap \TT$.
The inclusion $\WW \subseteq \UU^\perp \cap \TT$ is easy to check.
The other inclusion $\UU^\perp \cap \TT \subseteq \WW$ follows from $\TT=\UU*\WW$.
\end{proof}

We also need the following technical lemma,
which is a generalization of \cite[Lemma 3.7]{DIRRT} and \cite[Lemma 2.7]{A}.

\begin{lemma}\label{lemma:littleboat}
Let $\UU\in\tors \AA$ and $S\in\UU^\perp$ a brick. Set $\TT:=\T(\UU,S)$. 
\begin{itemize}
\item[(1)]
Every homomorphism $f \colon X\xto{}S$ in $\TT$ is zero or epic and satisfies $\Kernel f \in\TT$.
\item[(2)]
The brick $S$ belongs to $\simple(\WL(\TT))$.
\end{itemize}
\end{lemma}

\begin{proof}
(1)
We use induction on the length of $X$ in $\AA$.
If $X=0$ the claim is clear, so we assume $X \ne 0$.
Since $X\in\T(\UU,S)=\Filt(\UU,\Fac S)$,
we can take an exact sequence 
$0\to Y\xrightarrow{a}X\to X'\to0$ 
such that $Y$ is non-zero and belongs to $\UU$ or $\Fac S$. 
Then $Y \in \TT$ in both cases.

If $fa \colon Y\to S$ is non-zero, then $Y\in\Fac S$ holds, since $S \in \UU^\perp$.
There exists an epimorphism $g \colon S^n \to Y$, and the composite
$fag \colon S^n \to S$ is non-zero, and hence a split epimorphism of $S$.
Thus $f$ is also a split epimorphism and $\Kernel f \in\TT$ as desired.

On the other hand, if $fa \colon Y\to S$ is zero, 
then $Y \subseteq \Kernel f$, so we have a commutative diagram of exact sequences
\begin{align*}\xymatrix@R1.2em@C2em{
&&0\ar[d]&0\ar[d]\\
0\ar[r]&Y\ar[r]\ar@{=}[d]&\Kernel f\ar[r]\ar[d]&\Kernel f'\ar[r]\ar[d]&0\\
0\ar[r]&Y\ar[r]^a&X\ar[r]\ar[d]^f&X'\ar[r]\ar[d]^{f'}&0\\
&&S\ar@{=}[r]&S
}\end{align*}
By the induction hypothesis, $f'$ is either zero or epic, and satisfies $\Kernel f'\in\TT$.
Thus $f$ is either zero or epic, and satisfies $\Kernel f\in\TT$.

(2)
This follows immediately from (1).
\end{proof}


Thus $\simple(\WL(\TT))$ is given by the brick labeling of $\Hasse(\tors \AA)$.

\begin{proposition}\label{proposition:leftwidesubcat}
Let $\TT\in\tors \AA$ and set $\VV:=\TT \cap {^\perp}\WL(\TT)\in\tors\AA$.
Then we have $\simple(\WL(\TT))=\Label{[\VV,\TT]^+}=\Label{[0,\TT]^+}$.
\end{proposition}
\begin{proof}
First $\simple(\WL(\TT))=\Label{[\VV,\TT]^+}$ follows from Proposition \ref{proposition:serremutations} and Theorem \ref{theorem:joinintervals},
and $\Label{[\VV,\TT]^+} \subseteq \Label{[0,\TT]^+}$ is obvious.

Now it remains to show $\Label{[0,\TT]^+} \subseteq \simple(\WL(\TT))$.
Let $S \in \Label{[0,\TT]^+}$. Then there exists a Hasse arrow $\TT\xto{S}\UU$ in 
$\tors \AA$, and Proposition \ref{proposition:interval} implies $\TT=\T(\UU,S)$; 
hence $S\in\WL(\TT)$ by Lemma \ref{lemma:littleboat}.
\end{proof}

Now we get a characterization of wide intervals in terms of 
left and right wide subcategories.

\begin{theorem}\label{theorem:leftwidesubcat}
Let $[\UU,\TT]$ be an interval in $\tors \AA$ and set $\WW:=\UU^\perp\cap\TT$. Then the following conditions are equivalent:
\begin{itemize}
\item[(a)] $\WW\in\wide \AA$,
\item[(b)] $\WW\in\Serre(\WL(\TT))$,
\item[(c)] $\WW\in\Serre(\WR(\UU^\perp))$,
\item[(d)] $\WW=\WR(\UU^\perp)\cap\WL(\TT)$.
\end{itemize}
\end{theorem}
\begin{proof} 
$\text{(a)}\Rightarrow \text{(b)}$: 
By assumption $[\UU,\TT]$ is a wide interval, so 
$\WW=\Filt(\Label{[\UU,\TT]^+})$ follows from Theorem \ref{theorem:reduction}.
Therefore $\WW\in\Serre(\WL(\TT))$ according to Proposition \ref{proposition:leftwidesubcat}.

$\text{(b)}\Rightarrow \text{(a)}$: 
This is obvious, since by definition Serre subcategories are closed under extensions.

$\text{(a)}\Leftrightarrow \text{(c)}$: 
This follows by duality from $\text{(a)}\Leftrightarrow \text{(b)}$.

Thus it suffices to show $(\text{(b) and (c)}) \Leftrightarrow \text{(d)}$.

$(\text{(b) and (c)}) \Rightarrow \text{(d)}$:
This is deduced as 
$\WW \subseteq \WR(\UU^\perp) \cap \WL(\TT) \subseteq \UU^\perp \cap \TT = \WW$.

$\text{(d)} \Rightarrow \text{(b)}$:
Assume $\WW=\WR(\UU^\perp)\cap\WL(\TT)$.
Since $\WW=\WR(\UU^\perp)\cap\WL(\TT) \in \wide(\WL(\TT))$,
it remains to check that
$\WW$ is closed under taking subobjects in $\WL(\TT)$.
Let $X \in \WW$ and $X' \subseteq X$ satisfy $X' \in \WL(\TT)$.
As $X \in \WW \subseteq \UU^\perp$, we get $X' \in \UU^\perp$.
By assumption, $X' \in \TT$, so $X' \in \UU^\perp \cap \TT = \WW$ as desired.
Thus $\WW\in\Serre(\WL(\TT))$.

$\text{(d)} \Rightarrow \text{(c)}$:
This can be checked as $\text{(d)} \Rightarrow \text{(b)}$.
\end{proof}

Propositions \ref{proposition:serremutations}, 
\ref{proposition:leftwidesubcat}
and Theorem \ref{theorem:leftwidesubcat} yield the following property:

\begin{theorem}\label{theorem:power}
Let $\TT \in \tors \AA$. Taking labels gives a bijection
\begin{align*}
\{ \text{Hasse arrows in $\tors \AA$ starting at $\TT$} \} \to \simple(\WL(\TT)).
\end{align*}
Moreover, the map $\WW \mapsto \TT \cap {^\perp \WW}$ induces a bijection
\begin{align*}
\Serre(\WL(\TT)) \to \{ \VV \in \tors \AA \mid \text{$[\VV,\TT]$ is a wide interval}\}.
\end{align*}
\end{theorem}

We end this section applying our results to $\tau$-tilting theory.
For this purpose, we recall some related notions.

Let $A$ be a finite-dimensional algebra over a field $K$,
$M \in \mod A$, and $P \in \proj A$.
Then, the pair $(M,P)$ is called a \textit{support $\tau$-tilting pair}
if $(M,P)$ is $\tau$-rigid and $|M|+|P|=|A|$,
where $|\cdot|$ denotes the number of isoclasses of indecomposable direct summands.
Adachi--Iyama--Reiten \cite[Theorem 2.7]{AIR} showed that 
there exists a bijection from 
the set of basic support $\tau$-tilting pairs 
to the set of functorially finite torsion classes in $\mod A$,
given by $M \mapsto \Fac M$.

If two distinct support $\tau$-tilting pairs $(M,P) \ne (M',P')$ has
a common direct summand $(N,Q)$ with $|N|+|Q|=|A|-1$,
then we say that $(M',P')$ is a mutation of $(M,P)$.
In this case, $\Fac M' \subsetneq \Fac M$ or $\Fac M' \supsetneq \Fac M$ holds
\cite[Definition-Proposition 2.28]{AIR}.
The former case is called a left mutation, and the latter is called a right mutation.

For a fixed support $\tau$-tilting $(M,P)$,
\cite[Theorem 3.1]{DIJ} implies that 
any arrow starting at $\Fac M$ in $\Hasse(\tors \AA)$ 
comes from some left mutation of $(M,P)$;
more explicitly, 
if $\UU \in \tors \AA$ has an arrow $\Fac M \to \UU$ in $\Hasse(\tors \AA)$,
then there exists a left mutation of $(M',P')$ of $(M,P)$
satisfying $\Fac M'=\UU$.
Therefore the arrows starting at $\Fac M$ in $\Hasse(\tors \AA)$ 
bijectively correspond to the left mutations of $(M,P)$.
By Proposition \ref{proposition:leftwidesubcat},
the labels of the arrows starting at $\Fac M$ coincides with $\simple(\WL(\TT))$.
Thus we get the following result by using Theorem \ref{theorem:power}.

\begin{corollary}
Let $A$ be a finite-dimensional algebra over a field $K$ and 
$(M,P)$ be a support $\tau$-tilting pair in $\mod A$.
Consider the torsion class $\TT:=\Fac M$.
If $m$ is the number of left mutations of the support $\tau$-tilting pair $(M,P)$,
then $\simple(\WL(\TT))$ has exactly $m$ elements,
and there exist exactly $2^m$ torsion classes $\VV$
such that $[\VV,\TT]$ are wide intervals.
\end{corollary}

We remark that all such torsion classes $\VV$ are functorially finite
in $\mod A$ by \cite[Theorem 3.14]{J}.

\section{Widely generated torsion classes}\label{section:ITbijections}

In this section we consider widely generated torsion classes defined as follows:

\begin{definition}
Let $\TT \in \tors \AA$.
Then $\TT$ is called a \textit{widely generated torsion class}
if $\TT$ admits some $\WW \in \wide \AA$ such that $\TT=\T(\WW)$.
\end{definition}

By Proposition \ref{proposition:MS},
it is easy to see that $\TT$ is a widely generated torsion class
if and only if $\TT=\T(\WL(\TT))$.
We have more characterizations of widely generated torsion classes 
from the results in the previous section.

\begin{theorem}\label{theorem:widetorsionclass}
For $\TT\in\tors \AA$, the following conditions are equivalent:
\begin{itemize}
\item[(a)] $\TT$ is a widely generated torsion class.
\item[(b)] $\TT=\T(\WL(\TT))$.
\item[(c)] $\TT=\T(\LL)$ where $\LL:=\Label{[0,\TT]^+}$.
\item[(d)] For every $\UU\in\tors \AA$ such that $\UU\subsetneq\TT$, there exists a Hasse arrow $\TT\xto{}\UU'$ such that $\UU\subseteq\UU'$.
\end{itemize}
\end{theorem}

\begin{proof}
$\text{(a)}\Leftrightarrow\text{(b)}$ follows from Proposition \ref{proposition:MS} 
and $\text{(b)}\Leftrightarrow\text{(c)}$ is implied by Proposition \ref{proposition:leftwidesubcat}.

$\text{(c)}\Rightarrow\text{(d)}$: 
Let $\UU\in\tors \AA$ such that $\UU\subsetneq\TT$. 
Then there exists some $S\in\LL$ such that $S\notin\UU$ since $\TT=\T(\LL)$. 
From there, one can conclude $\UU\subseteq \UU':=\TT\cap{^\perp S}$; 
indeed for every $U \in \UU$ and $f \colon U \to S$, 
$f$ must be zero or epic by Lemma \ref{lemma:littleboat}, 
and $f=0$ since $S\notin\UU$.

$\text{(d)}\Rightarrow\text{(c)}$: The inclusion $\T(\LL)\subseteq\TT$ follows from $\LL\subseteq\TT$. For the other inclusion, suppose $\T(\LL)\subsetneq\TT$. Then by assumption there is a Hasse arrow $\TT\xto{S}\UU'$ with $\T(\LL)\subseteq\UU'$. Obviously we get $S\in\LL\subseteq\T(\LL)$, but it contradicts $S\in(\UU')^\perp\subseteq\T(\LL)^\perp$.
\end{proof}

\begin{remark}
The equivalences of (a), (c) and (d) above were also obtained 
in \cite[Subsection 3.2]{BCZ}
in terms of minimal extending modules and canonical join representations of torsion classes.
If the conditions above hold,
then $\TT=\Join_{S \in \LL}\T(S)$ is the canonical join representation of $\TT$.
\end{remark}

\begin{example}
Consider the following algebra $A$ appearing in \cite[Example 4.13]{A}:
\begin{align*}
K\left( \begin{xy}
( 0, 0) *+{1} = "1",
(16, 0) *+{2} = "2",
(32, 0) *+{3} = "3",
\ar@< 1mm>^{\alpha} "1";"2"
\ar@<-1mm>_{\beta} "1";"2"
\ar^{\gamma} "2";"3"
\end{xy}\right)/\langle \alpha \gamma \rangle.
\end{align*}
Let $\TT$ be the full subcategory
\begin{align*}
\add\left(
S_3, S_1,\begin{smallmatrix} \!\!1&\!\!&\!\!1\\ \!\!&\!\!2&\!\! \end{smallmatrix},
\begin{smallmatrix} \!\!1&\!\!&\!\!1&\!\!&\!\!1\\ \!\!&\!\!2&\!\!&\!\!2&\!\! \end{smallmatrix}, \ldots
\right),
\end{align*}
where $S_3$ is the simple projective module corresponding to the vertex 3, and 
$S_1,\begin{smallmatrix} \!\!1&\!\!&\!\!1\\ \!\!&\!\!2&\!\! \end{smallmatrix},
\begin{smallmatrix} \!\!1&\!\!&\!\!1&\!\!&\!\!1\\ \!\!&\!\!2&\!\!&\!\!2&\!\! \end{smallmatrix}, \ldots$ are all the indecomposable preinjective modules over the quotient algebra $A/\langle e_3 \rangle \cong K(1 \rightrightarrows 2)$.
Then one can check that $\TT$ is a torsion class in $\mod A$.

The simple projective module $S_3$ 
belongs to $\WL(\TT)$, and clearly it is simple in $\WL(\TT)$.
There exists no other simple object in $\WL(\TT)$;
indeed if $S \not \cong S_3$ is a simple object in $\WL(\TT)$,
then $S$ must be a preinjective $(A/\langle e_3 \rangle)$-module in $\TT$,
but one can easily check that 
any preinjective $(A/\langle e_3 \rangle)$-module cannot be in $\WL(\TT)$.
Thus $\WL(\TT)=\Filt S_3$, and $\TT \ne \T(\WL(\TT))$ follows.
The condition (b) in Theorem \ref{theorem:widetorsionclass} does not hold for 
the torsion class $\TT$.
Therefore $\TT$ is not a widely generated torsion class.

Set $\UU:=\TT \cap {^\perp S_3}$. Then the modules in $\UU$ are 
all the preinjective $(A/\langle e_3 \rangle)$-modules.
There exists a Hasse arrow $\TT \to \UU$,
and it is the unique Hasse arrow starting at $\TT$
by Proposition \ref{proposition:leftwidesubcat}.
Thus $\TT$ does not satisfy the condition (d) in Theorem \ref{theorem:widetorsionclass};
for example $\add S_3 \subsetneq \TT$ is a torsion class not contained in $\UU$.
\end{example}

\section*{Funding}

Sota Asai was supported by Japan Society for the Promotion of Science KAKENHI JP16J02249 and JP19K14500.

\section*{Acknowledgement}

The authors thank Aaron Chan, Laurent Demonet, Osamu Iyama, Gustavo Jasso and Jan Schr\"{o}er for kind instructions and discussions.

\end{document}